\documentclass[a4paper,11pt]{article}

\usepackage{amsmath,amssymb,amsthm,array,mathrsfs}
\usepackage{amssymb,latexsym, bbm,comment,epsfig}

\renewcommand{\theequation}                            
       {\mbox{\arabic{section}.\arabic{equation}}}

\newcommand{\origsetminus}{} \let\origsetminus=\setminus           
\renewcommand{\setminus}{\!\origsetminus\!}

{\theoremstyle{plain}
\newtheorem{definition}{Definition}[section]
\newtheorem{lemma}[definition]{Lemma}
\newtheorem{theorem}[definition]{Satz}

} {\theoremstyle{definition}
\newtheorem{example}[definition]{Example}
\newtheorem{remark}[definition]{Remark}
}

\renewcommand{\mathbb}{\mathbbm}                     
\renewcommand{\epsilon}{\varepsilon}                 
\renewcommand{\phi}{\varphi}
\renewcommand{\le}{\leqslant}
\renewcommand{\ge}{\geqslant}

\newcommand{\origfoo}{} \let\origfoo=\sqrt           
\renewcommand{\sqrt}[1]{\origfoo{#1}\;}

\renewcommand{\O}{{\mathcal O}}                      
\newcommand{\abs}[1]{\left\lvert #1 \right\rvert}    
\newcommand{\norm}[1]{\left\lVert #1 \right\rVert}   

\renewcommand{\mathcal}{\mathscr}          

\DeclareMathOperator{\R}{{\mathbb R}}                
\DeclareMathOperator{\Rp}{{\mathbb R}_+}             
\DeclareMathOperator{\N}{{\mathbb N}}                

\DeclareMathOperator{\Id}{ Id}                        

\newcommand{\A}{{\mathcal A}}

\DeclareMathOperator{\Borel}{{\mathfrak B}}

\newcommand{\scapro}[2]{\langle #1,#2\rangle}       
\DeclareMathOperator{\1}{\mathbbm 1}

\newcommand{\Z}{{\mathcal Z}}

\title{Stable cylindrical L{\'e}vy processes \\and
 the stochastic Cauchy problem}

\author{ Markus Riedle \\
Department of Mathematics\\
King's College\\
London WC2R 2LS\\
United Kingdom }

\begin{document}

\maketitle

\abstract{In this work, we consider the stochastic Cauchy problem driven by the canonical $\alpha$-stable cylindrical L{\'e}vy process.
This noise naturally generalises the cylindrical Brownian motion or  space-time Gaussian white noise. We derive a sufficient and necessary condition for the existence of the weak and mild solution of the stochastic Cauchy problem and establish the temporal irregularity of the solution.}\\


\section{Introduction}

One of the most fundamental stochastic partial differential equations is  a linear evolution equation perturbed by an additive noise
of the form 
\begin{align} \label{eq.Cauchy-intro}
dX(t)=AX(t)\,dt + dL(t)\quad\text{ for }t\in [0,T], 
\end{align}
where $A$ is the generator of a strongly continuous semigroup $(T(t))_{t\ge 0}$ on a Hilbert space $U$. If $L$ is the standard cylindrical Brownian motion in $U$ then there exists a weak, or equivalently mild, solution in $U$
of \eqref{eq.Cauchy-intro} if and only if
\begin{align}\label{eq.cond-intro}
\int_0^T \norm{T(s)}_{\rm{HS}}^2 \,ds<\infty,
\end{align}
where $\norm{\cdot}_{\rm{HS}}$ denotes the Hilbert-Schmidt norm; see
\cite[Th.7.1]{NeervenWeis}.
For the example of the stochastic heat equation, 
in which $A$ is chosen as the Laplace operator $\Delta$, this result implies that there exists a mild solution if and only if the spatial dimension equals one. 
A natural next generalisation step is to replace the cylindrical Brownian motion by an $\alpha$-stable noise. Like the cylindrical Brownian motion this noise does not exist as a genuine stochastic process in an infinite-dimensional space.  

In the present article, we consider equation \eqref{eq.Cauchy-intro}
driven by an $\alpha$-stable cylindrical noise and a generator $A$ allowing a spectral decomposition. For this purpose, we introduce the 
canonical $\alpha$-stable cylindrical L{\'e}vy process as a natural generalisation of the cylindrical Brownian motion. By evoking 
the recently introduced approach to stochastic integration for deterministic integrands with respect to cylindrical L{\'e}vy processes
in \cite{Riedle15}, we derive that there exists a mild (or weak) solution in $U$ if and only if
\begin{align}\label{eq.cond-intro-alpha}
\int_0^T \norm{T(s)}_{\rm{HS}}^\alpha \,ds<\infty.
\end{align}
Obviously, this result ``smoothly'' extends the equivalent condition 
\eqref{eq.cond-intro} in the Gaussian setting. We demonstrate that in the case of the stochastic heat equation this result leads to the sufficient and necessary condition
\begin{align}\label{eq.condition-heat-intro} 
 \alpha d <4
\end{align}
for the existence of a mild solution, where $d$ denotes the spatial dimension. As already observed in other examples of cylindrical L{\'e}vy processes as driving noise, we finish this note by establishing that the solution has highly irregular paths.

Equation \eqref{eq.Cauchy-intro} in Banach spaces with an $\alpha$-stable noise, or even slightly more general with a subordinated cylindrical Brownian motion,  has already been considered by Brze\'zniak and Zabczyk in \cite{BrzZab10}. However, their approach is based on embedding the underlying Hilbert space $U$ in a larger space such that the cylindrical noise becomes a genuine L{\'e}vy process. This leads to the fact that their condition for the existence of a solution not only 
lacks the necessity but also is in terms of the larger Hilbert space, which per se is not related to equation \eqref{eq.Cauchy-intro}.  Moreover, they only show that the paths of the solution are irregular in the sense that there does not exist a modification of the solution 
with c\`adl\`ag paths in $U$.

Another approach to generalise the model of the driving noise is based on the Gaussian space-time white noise leading to a {\em L{\'e}vy space-time white noise}; see Albeverio et al.\ \cite{Albeverio-etal} or Applebaum and Wu \cite{ApplebaumWu}. In this framework, equation 
\eqref{eq.Cauchy-intro} (and even with a multiplicative noise) driven by 
an $\alpha$-stable L{\'e}vy white noise is considered in the work \cite{Balan} by Balan.
We show that the $\alpha$-stable L{\'e}vy white noise in \cite{Balan} corresponds to a canonical $\alpha$-stable cylindrical L{\'e}vy process in the same way as it is known in the Gaussian setting (see \cite[Th.3.2.4]{KallianpurXiong}). However,  and very much in contrast to the Gaussian setting, it turns out that the corresponding cylindrical L{\'e}vy process is defined on a Banach space different from the underlying Hilbert space $U$. This leads to the new phenomena that the necessary and sufficient condition for the existence of a solution of the heat equation in the space-time white noise approach differs from our condition \eqref{eq.condition-heat-intro} 
in the cylindrical approach.

\section{The canonical  $\alpha$-stable cylindrical L{\'e}vy process}

Let $U$ be a separable Banach space with dual $U^\ast$. The dual pairing is denoted by
$\scapro{u}{u^\ast}$ for $u\in U$ and $u^\ast\in U^\ast$. For any $u_1^\ast,\dots, u_n^\ast\in U^\ast$ we define the projection
\begin{align*}
\pi_{u_1^\ast,\dots, u_n^\ast}\colon U\to \R^n,
\qquad \pi_{u_1^\ast,\dots, u_n^\ast}(u)=\big(\scapro{u}{u_1^\ast},
\dots, \scapro{u}{u_n^\ast}\big).\end{align*}
The Borel $\sigma$-algebra in $U$ is denoted by  $\Borel(U)$. For a subset
$\Gamma$ of $U^\ast$, sets of the form
 \begin{align*}
C(u_1^\ast,\dots ,u_n^\ast;B)&:= 
  \pi^{-1}_{u_1^\ast,\dots, u_n^\ast}(B),
\end{align*}
with $u_1^\ast,\dots, u_n^\ast\in \Gamma$ and $B\in \Borel(\R^n)$ are
called {\em cylindrical sets with respect to $\Gamma$}. The set of all these cylindrical sets is
denoted by $\Z(U,\Gamma)$; it is a $\sigma$-algebra if $\Gamma$ is finite and  it is an algebra otherwise. If $\Gamma=U^\ast$ we write $\Z(U):=\Z(U,U^\ast)$. 
A function $\mu:\Z(U)\to [0,\infty]$ is called a {\em cylindrical measure on
$\Z(U)$}, if for each finite subset $\Gamma\subseteq U^\ast$ the restriction of
$\mu$ to the $\sigma$-algebra $\Z(U,\Gamma)$ is a measure. A cylindrical
measure is called finite if $\mu(U)<\infty$ and a cylindrical probability
measure if $\mu(U)=1$. The characteristic
function $\phi_\mu:U^\ast\to {\mathbb C}$ of a finite cylindrical measure $\mu$ is defined by
\begin{align*}
 \phi_{\mu}(u^\ast):=\int_U e^{i\scapro{u}{u^\ast}}\,\mu(du)\qquad\text{for all }u^\ast\in  U^\ast.
\end{align*}

Let $(\Omega,\A,P)$ be a probability space. The space of equivalence classes of measurable functions $f\colon\Omega\to U$ is denoted by $L_P^0(\Omega;U)$ and it is equipped with the  topology of convergence in probability.
A {\em cylindrical random variable in $U$} is a linear and continuous mapping
\begin{align*}
 Z\colon U^\ast\to L_P^0(\Omega;\R).
\end{align*}
The characteristic function of a cylindrical random variable $Z$ is defined by
\begin{align*}
\phi_Z\colon U^\ast\to {\mathbb C}, \qquad \phi_Z(u^\ast)=E\left[ e^{i Z u^\ast}\right].
\end{align*}
If $C=C(u_1^\ast,\dots, u_n^\ast;B)$ is a cylindrical set for
$u^\ast_1,\dots, u^\ast_n\in U^\ast$ and $B\in \Borel(\R^n)$ we obtain a cylindrical probability measure $\mu$ by the prescription
\begin{align*}
  \mu(C):=P\big((Zu^\ast_1,\dots, Zu^\ast_n)\in B\big).
\end{align*}
We call $\mu$ the {\em cylindrical distribution of $Z$} and the
characteristic functions $\phi_\mu$ and $\phi_Z$ of $\mu$ and $Z$
coincide. 

A family $(Z(t):\,t\ge 0)$ of cylindrical random variables $Z(t)$ in $U$ is called
a {\em cylindrical process in $U$}. In our work \cite{DaveMarkus} with Applebaum, we extended the concept of
cylindrical Brownian motion to cylindrical L{\'e}vy processes:
\begin{definition}
A cylindrical process $(L(t):\, t\ge 0)$ in $U$ is called a {\em cylindrical L{\'e}vy process} if for each $n\in\N$ and any $u_1^\ast,\dots, u_n^\ast\in U$ we have that
\begin{align*}
\big( (L(t)u_1^\ast,\dots, L(t)u_n^\ast):\, t\ge 0\big)
\end{align*}
is a L{\'e}vy process in $\R^n$.
\end{definition}
The characteristic function of $L(t)$ for each $t\ge 0$ is of the form
\begin{align*}
\phi_{L(t)}\colon U^\ast \to {\mathbb C},\qquad  \phi_{L(t)}(u^\ast)&=\exp\big(t \Psi(u^\ast)\big),
\end{align*}
where $\Psi\colon U^\ast \to {\mathbb C}$ is called the {\em cylindrical symbol of $L$} and  is of the form
\begin{align*}
\Psi(u^\ast)=i a(u^\ast) -\tfrac{1}{2} \scapro{Qu^\ast}{u^\ast}
   +\int_U\left(e^{i\scapro{u}{u^\ast}}-1- i\scapro{u}{u^\ast}   \1_{B_{\R}}(\scapro{u}{u^\ast})\right)\nu(du).
\end{align*}
Here, $a\colon U^\ast\to\R$ is a continuous mapping with $a(0)=0$,  $Q\colon U^\ast \to U$ is a positive and symmetric operator and $\nu$ is a cylindrical measure on $\Z(U)$ satisfying
\begin{align*}
  \int_U \big(\scapro{u}{u^\ast}^2 \wedge 1\Big) \,\nu(du)<\infty
  \qquad\text{for all }u^\ast\in U^\ast.
\end{align*}
The characteristic function of $L$ is studied in detail in our work \cite{Riedle11}.

In this article we consider a specific example of a cylindrical L{\'e}vy process, 
which is obtained by the usual generalisation of the characteristic function of the standard normal distribution:
\begin{definition}
A cylindrical L{\'e}vy process $(L(t):\, t\ge 0)$ is called {\em canonical $\alpha$-stable} for  $\alpha\in (0,2)$ if its characteristic function is of the form
\begin{align*}
\phi_{L(t)}\colon U^\ast \to {\mathbb C}, \qquad
\phi_{L(t)}(u^\ast)=\exp\Big(-t\norm{u^\ast}^\alpha\Big).
\end{align*}
\end{definition}
Let $\mu$ be the cylindrical probability measure on $\Z(U)$ defined by 
the characteristic function 
\begin{align*}
\phi_{\mu}\colon U^\ast \to {\mathbb C}, \qquad
\phi_{\mu}(u^\ast)=\exp\Big(-\norm{u^\ast}^\alpha\Big),
\end{align*}
for some $\alpha\in (0,2)$. The cylindrical probability measure $\mu$ is called
the {\em canonical $\alpha$-stable cylindrical measure}. Bochner's theorem for cylindrical measures (\cite[Prop.IV.4.2]{Vaketal}) guarantees that $\mu$ exists. 
For, the function $\phi_\mu$ satisfies $\phi_\mu(0)=1$ and is continuous
and positive-definite (\cite[p.194]{Vaketal}). 
Two possible constructions of the canonical $\alpha$-stable L{\'e}vy process
such that its cylindrical distribution is given by $\mu$ are presented in Section \ref{se.representations}.

The cylindrical probability measure $\mu$ is  symmetric, i.e.\ it satisfies $\mu(C)=\mu(-C)$ for all $C\in \Z(U)$ and it is rotationally invariant, 
i.e.\ $\mu\circ M^{-1}=\mu$ for each linear unitary operator $M\colon U\to U$.
\begin{remark}
In \cite{Linde} among other publications, a symmetric cylindrical measure $\rho$ on $\Z(U)$ is called $\alpha$-stable  if there exists a measure space $(M,{\mathcal M},\sigma)$ 
and a linear, continuous operator $T\colon U^\ast \to L^\alpha_\sigma(M,{\mathcal M})$ such that  
\begin{align}\label{eq.def-Linde}
 \phi_\rho(u^\ast)=\exp\left(-\norm{Tu^\ast}_{L^\alpha_\sigma (M,{\mathcal M})}\right)
 \qquad\text{for all }u^\ast\in U^\ast.
\end{align}
For  the standard $\alpha$-stable cylindrical measure $\mu$, each 
image measure $\mu\circ \pi_{u^\ast}^{-1}$ is a stable measure 
on $\Borel(\R)$ for each $u^\ast\in U^\ast$. Thus, Theorem 6.8.5 in \cite{Linde} guarantees that $\mu$ 
also satisfies \eqref{eq.def-Linde}.
\end{remark}

In the case of a separable Hilbert space
we relate the characteristic function of the canonical $\alpha$-stable L{\'e}vy 
process with the well-known spectral representation on the 
sphere $S(\R^n):=\{\beta\in\R^n:\, \abs{\beta}=1\}$ of symmetric, rotationally invariant measures in $\R^n$:
\begin{lemma}\label{le.projected-spectral}
Let $U$ be a Hilbert space with orthonormal basis $(e_n)_{n\in\N}$ and
$L$ be the canonical $\alpha$-stable cylindrical L{\'e}vy process in $U$. Then
the characteristics of $L$ is given by $(0,0,\nu)$ with the 
cylindrical L{\'e}vy measure $\nu$ satisfying for all $n\in\N$: 
\begin{align*}
 \nu\circ\pi_{e_1,\dots, e_n}^{-1}(B)
 = \tfrac{\alpha}{c_\alpha} \int_{S(\R^n)}\lambda_n(d\xi)\int_0^\infty \1_{B}(r\xi)\frac{1}{r^{1+\alpha}}\, dr
 \qquad \text{ for }B\in\Borel(\R^n),
\end{align*}
 where $\lambda_n$ is  uniformly distributed on the sphere $S(\R^n)$ with
\begin{align*}
 \lambda_n\big(S(\R^n)\big)
=\frac{\Gamma(\tfrac{1}{2})\Gamma(\tfrac{n+\alpha}{2})}{\Gamma(\tfrac{n}{2})\Gamma(\tfrac{1+\alpha}{2})}
\end{align*}
and the constant $c_\alpha$ is defined in Theorem \ref{th.stable-R}.
\end{lemma}
\begin{proof}
For each $n\in\N$ and $\beta=(\beta_1,\dots, \beta_n)\in\R^n$ we obtain
\begin{align}\label{eq.char-project-e}
\phi_{L(1)e_1,\dots, L(1)e_n}(\beta)
&=\phi_{L(1)}\big(\beta_1e_1+\dots + \beta_ne_n\big)\notag\\
&= \exp\left(- \left( \norm{\beta_1e_1+\dots + \beta_ne_n}^2\right)^{\alpha/2}\right)
= \exp\left(-  \abs{\beta}^\alpha\right).
\end{align}
It follows that the distribution of the random vector $(L(1)e_1,\dots, L(1)e_n)$ is symmetric and rotationally invariant. As the L{\'e}vy measure of $(L(1)e_1,\dots, L(1)e_n)$ is given by $ \nu\circ\pi_{e_1,\dots, e_n}^{-1}$,
Theorem \ref{th.stable-R} implies 
\begin{align*}
 \nu\circ\pi_{e_1,\dots, e_n}^{-1}(B)
 = \tfrac{\alpha}{c_\alpha} \int_{S(\R^n)}\lambda_n(d\xi)\int_0^\infty \1_{B}(r\xi)\frac{1}{r^{1+\alpha}}\, dr,
\end{align*}
where $\lambda_n$ is  uniformly distributed on the sphere $S(\R^n)$
and  is defined by $\lambda_n(C)=c_\alpha (  \nu\circ\pi_{e_1,\dots, e_n}^{-1})((1,\infty)C)$ for all $C\in\Borel(S(\R^n))$. From Part (d) of
Theorem \ref{th.stable-R} and \eqref{eq.char-project-e} we deduce  that for each  $\xi_0\in S(\R^n)$ we have
\begin{align}\label{eq.1-and-spectral}
 1=\int_{S(\R^n)} \abs{\scapro{\xi_0}{\xi}}^\alpha\, \lambda_n(\xi)
 = r_n \int_{S(\R^n)} \abs{\scapro{\xi_0}{\xi}}^\alpha\, \lambda_n^1(\xi),
\end{align}
where $\lambda_n^1:=\tfrac{1}{r_n}\lambda_n$ and $r_n:= \lambda_n(S(\R^n))$.
Let $Y_n=(Y_{n,1},\dots, Y_{n,n})$ be uniformly distributed on $S(\R^n)$. By choosing $\xi_0=(1,0,
\dots, 0)$ in \eqref{eq.1-and-spectral} and 
applying Lemma~\ref{le.uniform-distribution} 
 we obtain
\begin{align*}
 1= r_n E[\abs{Y_{n,1}}^\alpha]
 =\frac{\Gamma(\tfrac{n}{2})\Gamma(\tfrac{1+\alpha}{2})}{\Gamma(\tfrac{1}{2})\Gamma(\tfrac{n+\alpha}{2})},
\end{align*}
which completes the proof. 
\end{proof}


\section{Two representations}\label{se.representations}

The first representation of the canonical $\alpha$-stable cylindrical L{\'e}vy process is by subordination. For this purpose, let $W$ be the standard cylindrical  Brownian  motion on the Banach space $U$, i.e.\ a cylindrical L{\'e}vy process with characteristics $(0,\Id,0)$ where 
$\Id$ denotes the identity on $U$. Its characteristic function is 
given by
\begin{align*}
\phi_{W(t)}\colon U^\ast\to {\mathbb C},
\qquad \phi_{W(t)}(u^\ast)=\exp\left(-\tfrac{1}{2}\norm{u^\ast}^2\right).
\end{align*}

\begin{lemma}\label{le.subordination}
Let $W$ be the standard cylindrical Brownian motion on a separable 
Banach space $U$ and let $\ell $ be an independent, real-valued $\alpha/2$-stable subordinator with 
L{\'e}vy measure $\nu_\ell(dy)=\tfrac{2^{\alpha/2}\alpha/2}{\Gamma(1-\alpha/2)} y^{-1-\alpha/2}\, dy$ for $\alpha\in (0,2)$. Then
\begin{align*}
L(t)u^\ast:= W\big(\ell(t)\big)u^\ast\qquad\text{for all }u^\ast\in U^\ast, 
\end{align*}
defines a canonical $\alpha$-stable cylindrical L{\'e}vy process $(L(t):\, t\ge 0)$
in $U$.
\end{lemma}
\begin{proof}
By using independence of $W$ and $\ell$, Lemma 3.8 in \cite{DaveMarkus} shows that $L$ is a cylindrical L{\'e}vy process. The characteristic function of the subordinator $\ell$ can be analytically continued, such that
\begin{align*}
E[\exp(-\beta \ell(t))] = \exp(-t \tau(\beta))\qquad\text{for all }\beta>0, 
\end{align*} 
where the Laplace exponent $\tau$ is given by
\begin{align*}
 \tau(\beta)= \int_0^\infty (1-e^{-\beta s})\, \nu_\ell(ds)=\beta^{\alpha/2};
\end{align*}
see \cite[Th.24.11]{Sato}.
Independence of $W$ and $\ell$ implies that for each $t\ge 0$ and $u^\ast \in U^\ast$
the characteristic function $\phi_{L(t)}$ of $L(t)$ is given by
  \begin{align*}
    \phi_{L(t)}(u^\ast)
    &= \int_0^\infty E\left[e^{i W(s)u^\ast}\right]\, P_{\ell (t)}(ds)\notag \\
    &= \int_0^\infty e^{-\tfrac{1}{2}s \norm{u^\ast}^2}\, P_{\ell (t)}(ds)\notag 
    = \exp\left(-t\tau\left(\tfrac{1}{2}\norm{u^\ast}^2\right)\right)\notag
     = \exp\big(-t \norm{u^\ast}^\alpha\big),
  \end{align*}
which shows that $L$ is canonical $\alpha$-stable. 
\end{proof}

The second representation of the canonical $\alpha$-stable cylindrical L{\'e}vy process is based on the approach by L{\'e}vy space-time white noise, as it is defined for example in \cite{Albeverio-etal} and \cite{ApplebaumWu}.
\begin{definition}\label{de.Levy-White-Noise}
Let $(M,{\mathcal M},\nu)$ be a $\sigma$-finite measure space. A L{\'e}vy white noise
on a $\sigma$-finite measure space $(E,{\mathcal E},\sigma)$ with intensity $\nu$
is a random measure $Y\colon {\mathcal E}\times \Omega\to \R$  of the form
\begin{align*}
Y(B)=W(B)+\int_{B\times M} a(x,y)\,N(dx,dy)+\int_{B\times M} b(x,y)\,
 (\sigma\otimes \nu)(dx,dy), 
\end{align*} 
where 
\begin{minipage}[t]{9cm}
\begin{enumerate}
\item[{\rm (1)}] $W\colon {\mathcal E}\times \Omega\to \R$ is 
a Gaussian white noise;
\item[{\rm (2)}] $N\colon ({\mathcal E}\otimes {\mathcal M})\times \Omega\to \N_0\cup\{\infty\}$ is a Poisson random measure on $E\times M$ with intensity measure $\sigma\otimes \nu$;
\item[{\rm (3)}] $a,b\colon E\times U\to \R$ are measurable functions.
\end{enumerate}
\end{minipage}
\end{definition}

In the above Definition \ref{de.Levy-White-Noise} we take $M=\R$,
${\mathcal M}=\Borel(\R)$  and 
\begin{align*}
\nu(dy)=\frac{1}{c} \frac{1}{\abs{y}^{\alpha+1}}\,dy
\qquad\text{for }c:=2\Gamma(\alpha)\cos(\tfrac{\pi\alpha}{2}).
\end{align*}
For some set $\O\subseteq \R^d$ let $N$ be a Poisson random measure on $\big([0,\infty)\times\O\big)\times \R$
with intensity measure $\sigma\otimes \nu$ where $\sigma:=dt\otimes dx$. 
We call the L{\'e}vy white noise $Y\colon \Borel([0,\infty)\times\O)\times\Omega\to [0,\infty)$ defined by
\begin{align*}
 Y(B)=\begin{cases}\displaystyle
 \int_{B\times \R} y \,N(dt,dx,dy), &\text{if }\alpha<1,\\
\displaystyle 
 \int_{B\times \R} y\, \big(N(dt,dx,dy)- dt\,dx\,\nu(dy)\big),
 &\text{if }\alpha\ge 1, 
\end{cases}
\end{align*}
 the {\em canonical $\alpha$-stable space-time L{\'e}vy white noise 
on $\O$}. 

\begin{lemma}\label{le.white-noise}
Let $Y$ be the canonical $\alpha$-stable L{\'e}vy space-time white noise on 
a set $\O\subseteq\R^d$ for $\alpha\in (1,2)$. Then there exists a canonical $\alpha$-stable cylindrical L{\'e}vy process $L$ in $L^{\alpha^\prime}(\O)$ for $\alpha^\prime:=\tfrac{\alpha}{\alpha-1}$ such that 
\begin{align*}
 L(t)\1_A= Y([0,t]\times A)\qquad\text{for all }t\ge 0 \text{ and bounded sets }
  A\in \Borel(\O). 
\end{align*}
\end{lemma}
\begin{proof}
Let $S(\O)$ be the space of simple functions in $L^\alpha(\O)$ of the form
\begin{align}\label{eq.u-simple}
 u^\ast:=\sum_{k=1}^n \beta_k\1_{A_k},
\end{align}
for some $\beta_k\in\Rp$ and disjoint sets $A_k\in \Borel(\O)$. For $u^\ast\in S(\O)$  and $t\ge 0$ 
we define
\begin{align}\label{de.Levy-Y}
L(t)u^\ast:= \sum_{k=1}^n \beta_k Y([0,t]\times A_k). 
\end{align}
Define the function $m_\beta\colon\R\to\R$ by $m_\beta(x)=\beta x$ for some $\beta\in \Rp$. 
Then by using the invariance $\beta^{-\alpha} (\nu\circ m_\beta^{-1})=\nu$, see  \cite[Th.14.3]{Sato},
the independence of the L{\'e}vy white noise and the symmetry of $\nu$, we obtain
\begin{align*}
&\phi_{L(t)}(u^\ast)\\
&= \prod_{k=1}^n \exp\left( \int_{[0,t]\times A_k\times \R}
\left(e^{i\beta_k y} -1 -i\beta_k y \1_{B_{\R}}(y)\right)\, \nu(dy)\,dx\, ds\right)\\
&= \prod_{k=1}^n \exp\left( t\left(\int_{\O} \1_{A_k}(x) \,dx \right)\beta_k^{\alpha} \int_{\R}
\left(e^{i y} -1 -i y \1_{B_{\R}}(\beta_k^{-\alpha} y)\right)\, \beta_k^{-\alpha} (\nu\circ m_{\beta_k}^{-1})(dy)\right)\\
&= \prod_{k=1}^n \exp\left( t\left(\int_{\O} \beta_k^{\alpha} \1_{A_k}(x) \,dx \right) \int_{\R}
\left(e^{i y} -1 -i y \1_{B_{\R}}( y)\right)\, \nu(dy)\right)\\
&= \exp\left( t\left(\int_{\O} \abs{u^\ast(x)}^{\alpha} \,dx \right) \int_{\R}
\left(e^{i y} -1 -i y \1_{B_{\R}}( y)\right)\, \nu(dy)\right).
\end{align*}
By applying Lemma 14.11 in \cite{Sato} we obtain $\phi_{L(t)}(u^\ast)=\exp(-t\norm{u^\ast}_{L^\alpha}^\alpha)$ for all simple functions $u^\ast\in S(\O)$. By using the linearity 
of $L(t)$ we derive that $L(t)\colon S(\O)\to L^0_P(\Omega)$ is a linear and continuous operator  which shows that $L$ can be continued to a linear and continuous operator  on 
$L^\alpha(\O)$ satisfying 
\begin{align*}
\phi_{L(t)}(u^\ast)=\exp(-t\norm{u^\ast}_{L^\alpha}^\alpha)
\qquad\text{for all }u^\ast\in L^\alpha(\O). 
\end{align*}
Let $0\le t_1\le t_2\le \cdots \le t_n$ and $u_j^\ast\in S(\O)$. By independence of 
the L{\'e}vy white noise for disjoint sets it follows that 
\begin{align*}
L(t_1)u_1^\ast,\big(L(t_2)-L(t_{1})\big)u_2^\ast  \dots, \big(L(t_n)-L(t_{n-1})\big)u_n^\ast
\end{align*}
are independent. By approximating an arbitrary function $u^\ast\in L^\alpha(\O)$ by simple 
functions it follows that $L$ has independent increments. Moreover, 
from the very definition in \eqref{de.Levy-Y} it follows that for each 
$u^\ast\in S(\O)$, the stochastic process $(L(t)u^\ast:\ t\ge 0)$ is a L{\'e}vy process in $\R$. 
Again, by approximating an arbitrary function $u^\ast\in L^\alpha(\O)$ by simple 
functions, the same conclusion holds for $(L(t)u^\ast:\, t\ge 0)$ and $u^\ast \in L^\alpha(\O)$. 
Together with the independent increments derived above, this implies by Corollary 3.8  in \cite{DaveMarkus} that $L$ is a cylindrical L{\'e}vy process in $L^{\alpha^\prime}(\O)$. 
\end{proof}

\begin{remark}
The canonical $\alpha$-stable space-time L{\'e}vy white noise corresponds to the noise considered in 
\cite{Balan} in the symmetric case. Although the construction in Lemma 
\ref{le.white-noise} follows the corresponding relation between space-time Gaussian white noise 
and cylindrical Brownian motion, the resulting cylindrical L\'evy process does not live
in a Hilbert space, such as $L^2(\O)$ for $\O\subseteq \R^d$. 
\end{remark}

\section{The stochastic Cauchy problem}

In this section we consider the stochastic Cauchy problem
\begin{align}\label{eq.Cauchy}
\begin{split}
 dX(t)&=AX(t)\,dt + \,dL(t)\qquad\text{for } t \in [0,T],\\
  X(0)&=x_0,
\end{split}
\end{align}
where $A$ is the generator of a strongly continuous semigroup 
$(T(t))_{t\ge 0}$ on a separable Hilbert space $U$ and $x_0\in U$ denotes the initial condition. The random noise $L$ is a canonical $\alpha$-stable cylindrical 
L{\'e}vy process as introduced in the previous section.

A theory of stochastic integration  for deterministic functions with respect to  cylindrical L{\'e}vy processes is introduced in our work \cite{Riedle15}. Based on this  theory, we obtain that if the stochastic convolution integral
\begin{align*}
 Y(t):=\int_0^t T(t-s)\,L(s) \qquad\text{for }t\in [0,T],
\end{align*}
exists, the stochastic process $(T(t)x_0+ Y(t):\, t\ge 0)$  can be considered as a mild solution. By a result in 
\cite{Umesh-diss} it also follows that the mild solution is a weak solution; 
however, this is of less concern in this work. 

More precisely, if  $L$ is a  cylindrical L{\'e}vy processes with  characteristics $(0,0,\nu)$, then a function $f\colon [0,T]\to L(U,U)$ is stochastically integrable if and only if 
\begin{align}\label{eq.int-condition}
\limsup_{m\to\infty} \sup_{n\ge m}\int_0^T \int_U \left(\sum_{k=m}^{n} 
\scapro{u}{f^\ast(s)e_k}^2\wedge 1\right)\,\nu(du)\,ds=0,
\end{align}
for an orthonormal basis $(e_k)_{k\in\N}$ of $U$; see \cite[Th.5.10]{Riedle15}.
In our case of a canonical $\alpha$-stable cylindrical L{\'e}vy process as driving noise
we obtain the following equivalent conditions:
\begin{theorem}\label{th.exist-solution}
Assume that there exist an orthonormal basis $(e_k)_{k\in\N}$ of $U$ 
and $(\lambda_k)_{k\in\N}\subseteq [0,\infty)$ 
with $T^\ast(t)e_k=e^{-\lambda_k t} e_k$ for all $t\ge 0$ and $k\in\N$. 
Then there exists a mild (and weak) solution of \eqref{eq.Cauchy} if and only if
\begin{align*}
\int_0^T \norm{T(s)}_{{\rm HS}}^{\alpha}\,ds<\infty. 
\end{align*}
\end{theorem}
\begin{proof}
Lemma \ref{le.projected-spectral} implies for each $m,n\in\N$ with $m\le n$  that
\begin{align*}
&\int_0^T \int_U \left(\sum_{k=m}^{n} 
\scapro{u}{T^\ast(s)e_k}^2\wedge 1\right)\,\nu(du)\,ds\\
&\qquad =\int_0^T \int_U \left(\sum_{k=m}^{n} e^{-2\lambda_k s}
\scapro{u}{e_k}^2\wedge 1\right)\,\nu(du)\,ds\\
&\qquad =\int_0^T \int_{\R^{n-m+1}} \left(\sum_{k=m}^{n} e^{-2\lambda_k s}
\beta_k^2\wedge 1\right)\,\nu\circ \pi_{e_m,\dots, e_n}^{-1}(d\beta)\,ds\\
&\qquad =\tfrac{\alpha}{c_\alpha}\int_0^T \int_{S(\R^{n-m+1})}\int_0^\infty
 \left(\sum_{k=m}^{n} e^{-2\lambda_k s} r^2\xi_k^2\wedge 1\right)\,\frac{1}{r^{1+\alpha}}\,dr \,\lambda_{n-m+1}(d\xi) \,ds\\
&\qquad =\tfrac{2}{c_\alpha (2-\alpha)}\int_0^T \int_{S(\R^{n-m+1})}
 \left(\sum_{k=m}^{n} e^{-2\lambda_k s} \xi_k^2\right)^{\alpha/2} \,\lambda_{n-m+1}(d\xi) \,ds\\
&\qquad =:I_{m,n}. 
\end{align*}
In the following, we establish that for all $m$, $n\in\N$ we have
\begin{align}\label{eq.I-both-sides}
\begin{split}
\tfrac{2}{c_\alpha (2-\alpha)} \int_0^T \left( \sum_{k=m}^n \norm{T^\ast(s)e_k}^2 \right)^{\alpha/2}  \!\! ds
 \le I_{m,n}
 &\le c_{n-m}\int_0^T \left( \sum_{k=m}^n \norm{T^\ast(s)e_k}^2 \right)^{\alpha/2}  \! \!ds,
\end{split}
\end{align}
where $c_{n-m}\to 1$ as $m,n\to\infty$. 

Define $\lambda_{n-m+1}^1:=\tfrac{1}{r_{n-m+1}}\lambda_{n-m+1}$ where $r_{n-m+1}:=\lambda_{n-m+1}\big(S(\R^{n-m+1})\big)$.
By applying Jensen's inequality to the concave function $\beta\mapsto \beta^{\alpha/2}$  it follows from Lemma \ref{le.uniform-distribution} that
\begin{align*}
&\int_0^T \int_{S(\R^{n-m+1})}
 \left(\sum_{k=m}^n e^{-2\lambda_k s} \xi_k^2\right)^{\alpha/2} \,\lambda_{n-m+1}(d\xi) \,ds\\
 &\qquad \le r_{n-m+1}\int_0^T \left( \sum_{k=m}^n e^{-2\lambda_k s} \int_{S(\R^{n-m+1})}
  \xi_k^2\,\lambda_{n-m+1}^1(d\xi)\right)^{\alpha/2}  \,ds\\
 &\qquad = r_{n-m+1} \int_0^T \left( \frac{1}{n-m+1}\sum_{k=m}^n e^{-2\lambda_k s} \right)^{\alpha/2}  \,ds\\
 &\qquad =\frac{r_{n-m+1}}{(n-m+1)^{\alpha/2}} \int_0^T \left( \sum_{k=m}^n \norm{T^\ast(s)e_k}^2 \right)^{\alpha/2}  \,ds.
\end{align*}
Consequently, we obtain the upper bound in \eqref{eq.I-both-sides} with $c_{n-m}:=\tfrac{r_{n-m+1}}{(n-m+1)^{\alpha/2}}$. Since $\frac{\Gamma(x+\beta)}{\Gamma(x)x^\beta}\to 1$ as $x\to \infty $, 
Lemma \ref{le.projected-spectral} implies $c_{n-m}\to 1$ as $m,n\to \infty$.

For establishing the lower bound, define $c_{m,n}(s):=e^{-2\lambda_m s} + \dots + e^{-2\lambda_n s}$. 
We again apply Jensen's inequality to the same concave function $\beta\mapsto \beta^{\alpha/2}$ 
but with respect to the discrete probability measure $\{c_{m,n}^{-1}(s)e^{-2\lambda_m s}, 
\dots, c_{m,n}^{-1}(s)e^{-2\lambda_n s}\}$. In this way, by applying 
Lemma \ref{le.projected-spectral} and \ref{le.uniform-distribution} we deduce
\begin{align*}
&\int_0^T \int_{S(\R^{n-m+1})}
 \left(\sum_{k=m}^n e^{-2\lambda_k s} \xi_k^2\right)^{\alpha/2} \,\lambda_{n-m+1}(d\xi) \,ds\\
 &\qquad \ge 
 \int_0^T \int_{S(\R^{n-m+1})} (c_{m,n}(s))^{\alpha/2}
 \sum_{k=m}^n \frac{e^{-2\lambda_k s}}{c_{m,n}(s)} \xi_k^\alpha  \,\lambda_{n-m+1}(d\xi) \,ds\\
 &\qquad =
r_{n-m+1} \int_0^T (c_{m,n}(s))^{\alpha/2}
 \sum_{k=m}^n \frac{e^{-2\lambda_k s}}{c_{m,n}(s)}\int_{S(\R^{n-m+1})}  \xi_k^\alpha  \,\lambda_{n-m+1}^1(d\xi) \,ds\\
  &\qquad = r_{n-m+1}\frac{\Gamma(\tfrac{1+\alpha}{2})\Gamma(\tfrac{n}{2})}{\Gamma(\tfrac{1}{2})\Gamma(\tfrac{n+\alpha}{2})}  \int_0^T 
  (c_{m,n}(s))^{\alpha/2}\, ds \\
 &\qquad =
 \int_0^T \left( \sum_{k=m}^n \norm{T^\ast(s)e_k}^2 \right)^{\alpha/2}  \,ds.
\end{align*}
An application of \cite[Th.5.10]{Riedle15}, as summarised in \eqref{eq.int-condition}, completes the proof.
\end{proof}

\begin{example}
We consider the stochastic heat equation on a bounded 
domain $\O\subseteq \R^d$ with smooth boundary $\partial \O$ for some $d\in\N$. In this case, the generator $A$ is given by the Laplace operator $\Delta$ on $U:=L^2(\O)$. Thus, there exists an orthonormal basis $(e_k)_{k\in\N}$ of $U$ consisting of eigenvectors of $A$. According to Weyl's law, the eigenvalues $\lambda_k$ satisfy $\lambda_k\sim c k^{2/d}$ for $k\to\infty$ and a constant $c>0$. Consequently, we can assume that
$
 \lambda_k=c_kk^{2/d}
$
for constants $c_k$ with $c_k\in [a,b]$ for all $k\in\N$ and 
$0<a<b$. 

By the integral test for convergence of series we obtain for each $s>0$ that
\begin{align*}
\norm{T(s)}_{{\rm HS}}^2
= \sum_{k=1}^\infty e^{-2 s c_k k^{d/2}}
\le \int_0^\infty e^{-2sa x^{d/2}}\,dx
= \frac{2\Gamma(\tfrac{d}{2})}{d(2a)^{d/2}s^{d/2}}.
\end{align*}
 Analogously, we conclude 
for each $s>0$ that
\begin{align*}
\norm{T(s)}_{{\rm HS}}^2
= \sum_{k=1}^\infty e^{-2 s c_k k^{d/2}}
\ge -1+ \int_0^\infty e^{-2sb x^{d/2}}\,dx
= -1+ \frac{2\Gamma(\tfrac{d}{2})}{d(2b)^{d/2}s^{d/2}}.
\end{align*}
Consequently, we can deduce from Theorem \ref{th.exist-solution}
that there exists a mild solution of \eqref{eq.Cauchy} for $A=-\Delta$ if and only if
$
\alpha d< 4. 
$
In accordance with other works, e.g.\ \cite{Balan} or \cite{Mytnik}, the smaller the stable index $\alpha$ is the larger dimensions $d$ can be chosen 
in the condition for guaranteeing the existence of a weak solution. This is due to the smoother trajectories of stable processes for smaller stable index $\alpha$. 

Note that the sufficiency of the condition $\alpha d< 4$ can also be derived from results in \cite{BrzZab10}. Due to the approach of embedding the cylindrical L\'evy process $L$ in a larger space in \cite{BrzZab10}, the derivation is less direct than here; see Corollary 6.5 in \cite{Riedle15}. However, the work \cite{BrzZab10} provides further results on the spatial regularity of the weak solution; see e.g.\ Theorem 5.14.

\end{example}

\section{Irregularities of the trajectories}

In the work \cite{BrzZab10} it was observed that the solution 
of \eqref{eq.Cauchy} does not have a modification with c\`adl\`ag paths in the underlying Hilbert space $U$. We strengthen this result that the solution does not even have 
a c\`adl\`ag modification in a much weaker sense:
\begin{theorem}\label{th.nocadlagcond}
Assume that there exist an orthonormal basis $(e_k)_{k\in\N}$ of $U$ 
and $(\lambda_k)_{k\in\N}\subseteq [0,\infty)$ 
with $T^\ast(t)e_k=e^{-\lambda_k t} e_k$ for all $t\ge 0$ and $k\in\N$.
Then there does not exists a modification $\widetilde{X}$ of the 
mild solution of \eqref{eq.Cauchy} such that for each $u^\ast\in U^\ast$ 
the stochastic process $(\scapro{\widetilde{X}(t)}{u^\ast}:\, t\in [0,T])$ 
has c{\`a}dl{\`a}g paths.
\end{theorem}
\begin{proof} (The proof is based on ideas from \cite{LiuZhai}). 
For every $n\in \N$ and $t\in [0,T]$  define $L_n(t):=\big(L(t)e_1,\dots, L(t)e_n\big)$ and $X_n(t):=\big(\scapro{X(t)}{e_1},\dots, \scapro{X(t)}{e_n}\big)$. 
As
\begin{align*}
 \scapro{X(t)}{e_k}= \scapro{x_0}{e_k}e^{-\lambda_k t} + \int_0^t e^{-\lambda_k (t-s)}\, d(L(s)e_k),
\end{align*}
it follows that $X_n$ is the solution of the stochastic differential equation 
\begin{align*}
dX(t)= -DX(t)\,dt + dL_n(t)\qquad\text{for }t\ge 0, 
\end{align*}
where $D\in \R^{n\times n}$ is a diagonal matrix with entries $\lambda_1,\dots, \lambda_n$. We conclude that the $n$-dimensional processes $(X_n(t):\,t\in [0,T])$ and $(L_n(t):\,t\in [0,T])$  jump at the same time by the same magnitude, which implies
\begin{align*}
  \sup_{t\in [0,T]}\abs{\Delta L_n(t)}^2
  =  \sup_{t\in [0,T]}\abs{\Delta X_n(t)}^2
  \le 4 \sup_{t\in [0,T]}\abs{X_n(t)}^2,
\end{align*}
where $\Delta f(t):=f(t)-f(t-)$ for c{\`a}dl{\`a}g functions $f\colon [0,T]\to \R^n$.
It follows that
\begin{align}\label{eq.sum-and-Levy-n}
 P\left(\sup_{t\in [0,T]} \sum_{k=1}^\infty \scapro{X(t)}{e_k}^2<\infty\right)
& =   P\left(\sup_{n\in\N}\sup_{t\in [0,T]} \sum_{k=1}^n \scapro{X(t)}{e_k}^2<\infty\right)\notag\\
&=\lim_{c\to\infty}   P\left(\sup_{n\in\N}\sup_{t\in [0,T]} \sum_{k=1}^n \scapro{X(t)}{e_k}^2\le\frac{1}{4}c^2\right)\notag\\
&=\lim_{c\to\infty} \lim_{n\to\infty}  P\left(\sup_{t\in [0,T]} \sum_{k=1}^n \scapro{X(t)}{e_k}^2\le\frac{1}{4}c^2\right)\notag\\
&=\lim_{c\to\infty} \lim_{n\to\infty}  P\left(\sup_{t\in [0,T]} \abs{X_n(t)}^2\le\frac{1}{4}c^2\right)\notag\\
&\le\lim_{c\to\infty} \lim_{n\to\infty}  P\left(\sup_{t\in [0,T]} \abs{\Delta L_n(t)}^2\le c^2\right)\notag\\
&=\lim_{c\to\infty} \lim_{n\to\infty} \exp\left(-T \nu_n\Big(\{\beta\in\R^n:\abs{\beta}>c\}\Big)\right),
\end{align}
where $\nu_n$ denotes the L{\'e}vy measure of the $\R^n$-valued L{\'e}vy process $L_n$. Since
$\nu_n=\nu\circ\pi_{e_1,\dots, e_n}^{-1}$ due to \cite[Th.2.4]{DaveMarkus}, we obtain for every $n\in\N$ by Lemma \ref{le.projected-spectral} that   
\begin{align*}
&\big(\nu\circ \pi_{e_1,\dots, e_n}^{-1}\big)\big(\{\beta\in\R^n:\, \abs{\beta}\ge c\}\big)\\
&\qquad =\frac{\alpha}{c_\alpha}\int_{S(\R^n)} \int_0^\infty \1_{\{\beta\in\R^n:\abs{\beta}\ge c\}}(r\xi)\frac{1}{r^{1+\alpha}}\,dr\, \lambda_n(d\xi)\\
&\qquad =\frac{\alpha}{c_\alpha} \int_{S(\R^n)} \int_c^\infty \frac{1}{r^{1+\alpha}}\,dr\,\lambda_n(d\xi)\\
&\qquad= \frac{1}{c_\alpha c^{\alpha}} \lambda_n(S(\R^n))\\
&\qquad =\frac{1}{c_\alpha c^{\alpha}} \frac{\Gamma(\tfrac{1}{2})\Gamma(\tfrac{n+\alpha}{2})}{\Gamma(\tfrac{n}{2})\Gamma(\tfrac{1+\alpha}{2})}.
\end{align*}
Since $\frac{\Gamma(m+\beta)}{\Gamma(m)m^\beta}\to 1$ as $m\to \infty $ 
for all $\beta\in\R$, we obtain from \eqref{eq.sum-and-Levy-n} that 
\begin{align*}
    P\left(\sup_{t\in [0,T]} \sum_{k=1}^\infty \scapro{X(t)}{e_k}^2<\infty\right)
    =0.
\end{align*}
An application of Theorem 2.3 in \cite{PeszatZab12} completes the proof.
\end{proof}

\setcounter{section}{0}
\renewcommand{\thesection}{\Alph{section}} 
\section{Appendix}

For the  present work it is essential to know which constants in the various
presentations of the characteristic function of a multi-dimensional $\alpha$-stable distribution depend on the dimension. Thus,  we present the following well-known theorem with the constants given explicitly: 
%
%
\begin{theorem}\label{th.stable-R}
Let $\mu$  be an infinitely divisible probability measure on $\Borel(\R^n)$
with characteristics $(0,0,\nu)$ and define
$\lambda_n(B)=c_\alpha \nu\big((1,\infty)B\big)$ 
for $B\in\Borel(S(\R^n))$ and $\alpha\in (0,2)$ where
\begin{align*}
c_\alpha:=\begin{cases}
 -\alpha\cos(\tfrac{1}{2}\alpha \pi)\Gamma(-\alpha),& \text{if }\alpha\neq 1,\\
 \frac{\pi}{2}\alpha, &\text{if }\alpha=1.
\end{cases}
\end{align*}
Then the following are equivalent:
\begin{enumerate}
\item[\rm{(a)}] $\mu$ is symmetric, rotationally invariant and $\alpha$-stable;
\item[\rm{(b)}] the L{\'e}vy measure $\nu$ is of the form
\begin{align*}
 \nu(B)
 = \tfrac{\alpha}{c_\alpha} \int_{S(\R^n)}\lambda_n(d\xi)\int_0^\infty \1_{B}(r\xi)\frac{1}{r^{1+\alpha}}\, dr \qquad\text{for all $B\in\Borel(\R^n)$,}
\end{align*}
and $\lambda_n$ is uniformly distributed on the sphere  $S(\R^n)$.
\item[\rm{(c)}] the characteristic function $\phi_{\mu}\colon \R^n\to{\mathbb C}$ of $\mu$ is of the form
\begin{align*}
 \phi_\mu (\beta)=
 \exp\left(-\int_{S(\R^n)} \abs{\scapro{\beta}{\xi}}^\alpha \, \lambda_n(d\xi)\right),
\end{align*}
and $\lambda_n$ is uniformly distributed on the sphere  $S(\R^n)$.
\item[\rm{(d)}] the characteristic function $\phi_{\mu}\colon \R^n\to{\mathbb C}$ of $\mu$ is of the form
\begin{align*}
\phi_{\mu}(\beta)= \exp\Big(- d_\alpha \abs{\beta}^\alpha\Big),
\end{align*}
where $d_\alpha:= \int_{S(\R^n)} \abs{\scapro{\xi_0}{\xi}}^\alpha \, \lambda_n(d\xi)$ for an arbitrary fixed vector $\xi_0\in S(\R^n)$.
\end{enumerate}
\end{theorem}
\begin{proof}
The proof follows from the Theorems 14.2, 14.10, 14.13, 14.14 and their proofs in \cite{Sato}.

%
%
%
\end{proof}

\begin{lemma}\label{le.uniform-distribution}
Let $(Y_1,\dots, Y_n)$  be uniformly distributed on the sphere $S(\R^{n})$
for some $n\in\N$. Then we have for $p\in (0,2)$ and each $k\in\{1,\dots, n\}$ that
\begin{align*}
E[Y_k^2]=\frac{1}{n},\qquad E[\abs{Y_k}^p]= \frac{\Gamma(\tfrac{n}{2})\Gamma(\tfrac{1+p}{2})}{\Gamma(\tfrac{1}{2})\Gamma(\tfrac{n+p}{2})}
\end{align*}
\end{lemma}
\begin{proof}
Since $Y_1^2+\cdots + Y_n^2=1$  we have $E[Y_k^2]=\tfrac{1}{n}$ for all $k\in \{1,\dots, n\}$ due to symmetry. As $Y_{k}^2$ is distributed according to the Beta distribution with parameters 
$a:=\tfrac{1}{2}$ and $b:=\tfrac{n-1}{2}$, see \cite{FangKotzNG}, we obtain
\begin{align*}
E[\abs{Y_{k}}^p]
=E[\abs{Y_{k}^2}^{p/2}]
&=\frac{1}{B(a,b)}\int_0^1 
 x^{a+p/2-1}(1-x)^{b-1}\,dx\\
& = \frac{B(a+p/2,b)}{B(a,b)}
= \frac{\Gamma(\tfrac{n}{2})\Gamma(\tfrac{1+p}{2})}{\Gamma(\tfrac{1}{2})\Gamma(\tfrac{n+p}{2})},
\end{align*}
which completes the proof.
\end{proof}

\noindent\textbf{Acknowledgments:} the author would like to thank Ilya Pavlyukevich to 
initiate the study of the topic of this work
and Tomasz Kosmala and Umesh Kumar for proofreading.

\end{document}